\documentclass[12pt,a4paper,psamsfonts]{amsart}
\usepackage{amssymb,amscd,amsxtra,calc}
\usepackage{cmmib57}

\theoremstyle{plain}
    \newtheorem{thm}{Theorem}[section]

    \newtheorem{theorem}[thm]{Theorem}

\theoremstyle{definition}

\theoremstyle{remark}

\title[Backtracking Gradient Descent allowing unbounded learning rates]{Backtracking Gradient Descent allowing unbounded learning rates}

\author{Tuyen Trung Truong}
\address{Matematikk Institut, Universitetet i Oslo, Blindern, 0851 Oslo, Norway}
\email{tuyentt@math.uio.no}
\thanks{}
\date{\today}
\begin{document}
\begin{abstract}
In unconstrained optimisation on an Euclidean space, to prove convergence in Gradient Descent processes (GD) $x_{n+1}=x_n-\delta _n \nabla f(x_n)$ it usually is required that the learning rates $\delta _n$'s are bounded: $\delta _n\leq \delta $ for some positive $\delta $. Under this assumption, if the sequence $x_n$ converges to a critical point $z$, then with large values of $n$ the update will be small because $||x_{n+1}-x_n||\lesssim ||\nabla f(x_n)||$. This may also force the sequence to converge to a bad minimum. If we can allow, at least theoretically, that the learning rates $\delta _n$'s are not bounded, then we may have better convergence to better minima. 

A previous joint paper by the author showed convergence for  the usual version of Backtracking GD under very general assumptions on the cost function $f$. In this paper, we allow the learning rates $\delta _n$ to be unbounded, in the sense that there is a function $h:(0,\infty)\rightarrow (0,\infty )$ such that $\lim _{t\rightarrow 0}th(t)=0$ and $\delta _n\lesssim \max \{h(x_n),\delta \}$ satisfies Armijo's condition for all $n$, and prove convergence under the same assumptions as in the mentioned paper. It will be shown that this growth rate of $h$ is best possible if one wants convergence of the sequence $\{x_n\}$. 

A specific way for choosing $\delta _n$ in a discrete way connects to Two-way Backtracking GD defined in the mentioned paper. We provide some results which either improve or are implicitly contained in those in the mentioned paper and another recent paper on avoidance of saddle points.  
\end{abstract}
\maketitle

  
This short note is an addendum to our previous joint work \cite{truong-nguyen}. Hence we will keep it concise and refer the readers to the mentioned paper and references therein for historical details, results and terminologies. 

 We consider the problem of finding minimum of a $C^1$ function $f:\mathbb{R}^k\rightarrow \mathbb{R}$. 

A popular numerical way to solve this problem is to apply Gradient Descent processes: $x_{n+1}=x_n-\delta _n\nabla f(x_n)$, where $\delta _n>0$ (learning rates) are appropriately chosen.  Usually, it is assumed that $\delta _n$'s are bounded, i.e. there is a $\delta >0$ so that $\delta _n\leq \delta $ for all $n$. This may make the process to converge to bad critical points, since when $x_n$ is close to a critical point $z$ the change $x_{n+1}-x_n$ is bounded by $||\nabla f(x_n)||$ which is small. 

The boundedness of learning rates appear in many common versions of GD, including Standard GD (where $\delta _n$'s are independent of $n$), Diminishing Learning rates (when $\delta _n$ is assumed to converge to $0$) and also Backtracking GD. It may be helpful if we can allow learning rates to be unbounded, at least theoretically. In this note we provide one way to do so for Backtracking GD. 

We now recall the definition of Backtracking GD in the most basic form. We choose $0<\alpha ,\beta <1$ and $\delta _0>0$. For each $x\in \mathbb{R}^k$, we choose $\delta (x)$ to be the largest number $\delta$ in the discrete set $\{\beta ^n\delta _0: ~n=0,1,2,\ldots \}$ which satisfies Armijo's condition (see \cite{armijo})
\begin{eqnarray*}
f(x-\delta \nabla f(x))-f(x)\leq -\alpha \delta ||\nabla f(x)||^2. 
\end{eqnarray*}
The corresponding update rule for Backtracking GD is $x_{n+1}=x_n-\delta (x_n)\nabla f(x_n)$. In \cite{truong-nguyen} we proved convergence of Backtracking GD for cost functions $f$ having at most countably many critical points. This class of cost functions includes all Morse functions. Note that for Backtracking GD, the sequence of learning rates $\delta (x_n)$ are bounded from above by $\delta _0$. 

We now introduce a new version of Backtracking GD where we allow the learning rates $\delta (x_n)$ to not be bounded from above. 

{\bf Unbounded Backtracking GD.} Let $f$ be a $C^1$ function. Fix $0<\alpha ,\beta <1$ and $\delta _0 >0$. We choose $\delta (x)$ as in the Backtracking GD procedure. Fix a function $h:(0,\infty )\rightarrow (0,\infty)$ such that $\lim _{t\rightarrow 0}th(t)=0$. We choose $\hat{\delta}(x)$ any function satisfying $\delta (x)\leq \hat{\delta}(x)\leq h(||\nabla f(x)||)$  and Armijo's condition $f(x-\hat{\delta}(x)\nabla f(x))-f(x)\leq -\alpha \hat{\delta}(x)||\nabla f(x)||^2$, for all $x\in \mathbb{R}^k$. Choose a random point $x_0$. The update rule for Unbounded Backtracking GD is as follows: 
\begin{eqnarray*}
x_{n+1}=x_n-\hat{\delta}(x)\nabla f(x). 
\end{eqnarray*}

We have the following result, which is a generalisation of Theorem 2.1 in \cite{truong-nguyen}. 
\begin{theorem} Assume that $f$ is $C^1$, and $\{x_n\}$ is a sequence constructed by the Unbounded Backtracking GD procedure. Then: 

1) Any cluster point of $\{x_n\}$ is  a critical point of $f$. 

2) Either $\lim _{n\rightarrow\infty}f(x_n)=-\infty$ or $\lim _{n\rightarrow\infty}||x_{n+1}-x_n||=0$. 

3) Assume that $f$ has at most countably many critical points. Then either $\lim _{n\rightarrow\infty}||x_n||=\infty$ or $\{x_n\}$ converges to a critical point of $f$. 

4)  More generally, assume that the set of critical points of $f$ has a bounded connected component $A$. Let $B$ be the set of cluster points of $\{x_n\}$. If $B\cap A\not=\emptyset$, then $B$ is connected and $B\subset A$.  
\label{TheoremUnboundedBacktrackingGD}\end{theorem}
\begin{proof}
1) Let $K$ be a compact set for which $\inf _{x\in K}||\nabla f(x)||>0$. Then $\inf _{x\in K}\hat{\delta}(x)\geq \inf _{x\in K}\delta (x)>0$, the latter can be shown as in \cite{truong-nguyen}. Having this property, we can prove as in \cite{truong-nguyen} (see \cite{bertsekas}) that any cluster point of $\{x_n\}$ is a critical point of $f$. 

2) By Armijo's condition we have
\begin{eqnarray*}
f(x_{n+1})-f(x_n)\leq -\alpha \hat{\delta}(x_n)||\nabla f(x_n)||^2
\end{eqnarray*}
for all $n$. Hence either $\lim _{n\rightarrow\infty}f(x_n)=-\infty$ or $\lim _{n\rightarrow\infty}f(x_n)$ exists as a finite number. In that case, summing over $n$, we obtain the well-known estimate (see \cite{poljak-tsypkin}):
\begin{eqnarray*}
\sum _n\hat{\delta} (x_n)||\nabla f(x_n)||^2<\infty. 
 \end{eqnarray*}
 In particular, $\lim _{n\rightarrow \infty}\hat{\delta} (x_n)||\nabla f(x_n)||^2=0$. 
 
 For any $\epsilon >0$, we consider 2 sets: $C_1(\epsilon )=\{n\in \mathbb{N}: ~||\nabla f(x_n)||\leq \epsilon \}$ and $C_2(\epsilon )=\{n\in \mathbb{N}: ~||\nabla f(x_n)||>\epsilon \}$.  For $n\in C_1(\epsilon )$, using the assumption that $\lim _{t\rightarrow\infty}th(t)=0$ and that $\hat{\delta} (x_n)\leq h(||\nabla f(x_n)||)$, we obtain that
 \begin{eqnarray*}
 ||x_{n+1}-x_n||=\hat{\delta}(x_n)||\nabla f(x_n)||\leq h(||\nabla f(x_n)||) ||\nabla f(x_n)||,
 \end{eqnarray*}
must be small when $\epsilon $ is small. 

For $n\in C_2(\epsilon )$, we have
\begin{eqnarray*}
||x_{n+1}-x_n||=\hat{\delta}(x_n)||\nabla f(x_n)||\leq \hat{\delta}(x_n)||\nabla f(x_n)||^2/\epsilon, 
\end{eqnarray*}
which - for a fixed $\epsilon >0$ - must be small when $n$ large enough, because $\lim _{n\rightarrow\infty}\hat{\delta} (x_n)||\nabla f(x_n)||^2=0$. 

Combining these estimates, we obtain: $\lim _{n\rightarrow\infty}||x_{n+1}-x_n||=0$. 

3) and 4) follows from 1) and 2) by using the real projective space $\mathbb{P}^k$ as in \cite{truong-nguyen}, by using a result on convergence in compact metric spaces in \cite{asic-adamovic}. 
\end{proof}

{\bf Remarks.} The statement of 4) here is equivalent to that of Theorem 2.1 in \cite{truong-nguyen}, but stated in a form more convenient to apply. 

If we assume that the sequence $\{x_n\}$ converges, then $\lim _{n\rightarrow\infty}||x_{n+1}-x_n||=0$, and hence $\lim _{n\rightarrow\infty}\delta _n||\nabla f(x_n)||=0$. Thus the growth rate $\lim _{t\rightarrow 0}th(t)=0$ which we require is best possible if we want the process to converge. (On the other hand, if one chooses $\delta _n$ too small, so to make the condition $\lim _{n\rightarrow\infty}\delta _n||\nabla f(x_n)||=0$, then the limit point - if exists - may not be a critical point.)

If $x_n$ is near a critical point $z$ where the gradient is very flat, for example $\nabla f$ is Lipschitz continuous near $z$ with a very small Lipschitz constant $L(z)$, then the update $x_{n+1}=x_n-\delta _n\nabla f(x_n)$ is very small when $\delta _n$ is bounded. However, here we can take $\hat{\delta} (x_n)$ in the order of $1/L(z)$, which is big, and make big step and maybe can escape the point $z$ and go to another better critical point. (Of relevance is the Capture Theorem in \cite{bertsekas} which asserts that if $z_0$ is close enough to a local minimum and the learning rates are bounded, then the sequence $z_{n+1}=z_n-\delta _n\nabla f(z_n)$ cannot escape $z_0$.)

We can also prove an Unbounded version of Theorem 2.7 (Inexact Backtracking GD ) in \cite{truong-nguyen}, of Theorems 1.1 and 1.3 (avoidance of saddle points) in \cite{truong}.  Analog  results for Unbounded Backtracking versions of Momentum and NAG are also available. We end this note stating and proving some results which either improve or implicitly contained in those in  \cite{truong-nguyen, truong}. 

One way to make a discrete construction of Unbounded Backtracking GD is as follows. At Step n, we choose $\hat{\delta} (x_n)$ by starting with $\delta  =\delta _0$. If $\delta$ does not satisfy Armijo's condition, then we reduce it by a factor of $\beta$ as in the basic version of Backtracking GD. On the other hand, if $\delta$ does satisfy Armijo's condition, we multiply it by $1/\beta$ while Armijo's condition and $\delta \leq h(||\nabla f(x_n)||)$ are both still satisfied. $\hat{\delta} (x_n)$ is the final value of $\delta$. This construction is similar to Two-way Backtracking GD defined in \cite{truong-nguyen}, where the only differences are that we start with $\delta =\delta (x_{n-1})$, and we bound $\delta $ not by $h(||\nabla f(x_n)||)$ but with $\delta _0$. The following result, which proves convergence of Two-way Backtracking GD under some constraints, was not stated in \cite{truong-nguyen} but the proof follows easily from what presented there, by combining with arguments in \cite{truong}. The theorem is valid under assumptions (the same as those required in Theorem 1.3 in \cite{truong} about avoiding saddle points) including the case where $f$ is in $C^2$ or $C_L^{1,1}$. 

\begin{theorem}
Let $f:\mathbb{R}^k\rightarrow \mathbb{R}$ be a $C^1$ function. Assume that there are continuous functions $r,L:\mathbb{R}^k\rightarrow (0,\infty)$ such that for all $x\in \mathbb{R}^k$, the gradient $\nabla f$ is Lipschitz continuous on $B(x,r(x))$ with Lipschitz constant $L(x)$. Then all conclusions in Theorem \ref{TheoremUnboundedBacktrackingGD} are satisfied for the sequence $\{z_n\}$. 
\label{TheoremTwowayBacktrackingGD}\end{theorem}
\begin{proof}
 As observed in \cite{truong-nguyen}, the key is to prove the following two statements: 

i) If $K\subset \mathbb{R}^k$ is a compact set, then $\inf _{n: ~z_n\in K}\delta (z_n)>0$. This is satisfied under our assumption, since in fact it can be checked that $\delta (z_n)\geq  \min \{\beta /L(z_{n}),\beta r(z_{n})/||\nabla f(z_n)||,\delta _0\}$ and hence
\begin{eqnarray*}
\delta (z_n)\geq \inf _{z\in K} \min \{\beta /L(z),\beta r(z)/||\nabla f(z)||,\delta _0\}
\end{eqnarray*}
for all $n$, and that all functions $r,L,||\nabla f||$ are continuous, and $r,L>0$. Then we have that any cluster point of $\{z_n\}$ is  a critical point of $f$. 

ii) $\sup _{n}\delta (z_n)<\infty$. This is satisfied automatically since by construction $\delta (z_n)\leq \delta _0$ for all $n$. 
\end{proof}

When implementing Backtracking GD in DNN, \cite{truong-nguyen} proposed that we only apply it at several first iterations for each epoch, then after that use the Standard GD, unless when the cost function increases - at that point we apply Backtracking GD once, then repeat Standard GD. The following result, which justifies this practice, was not stated there, but again can easily be proven.

\begin{theorem} Let $f:\mathbb{R}^k\rightarrow \mathbb{R}$ be $C^1$ function. Fix a positive integer $N$. Given an initial point $x_0\in \mathbb{R}^k$. Assume that we construct a sequence $x_{n+1}=x_n-\delta (x_n)\nabla f(x_n)$. Assume that for every $n$, either $\delta (x_n)$ is constructed using Backtracking GD, or $\delta (x_n)=\delta (x_{n-1})$. Assume also that for every $n$, and at least one of $\delta (x_n),\delta (x_{n+1}),\ldots ,\delta (x_{n+N})$ is updated using Backtracking GD.  Moreover, assume that $f(x_{n+1})\leq f(x_n)$ for all $n$.  Then all conclusions of Theorem \ref{TheoremUnboundedBacktrackingGD} are satisfied for the sequence $\{x_n\}$.  
\label{Theorem}\end{theorem}
\begin{proof}
The proof of the theorem is similar to that of Theorem \ref{TheoremTwowayBacktrackingGD}, by observing that if  a subsequence $\{x_{n_k}\}$ converges and $\lim _{k\rightarrow\infty}||\nabla f(x_{n_k})||>0$, then the same is true for any other subsequence $x_{n'_{k}}$ with $|n_k-n'_{k}|\leq N$ for all $k$. Now we can choose such a subsequence $\{n'_k\}$ so that the update of $\delta (n'_k)$ is given by Backtracking GD, by using the assumptions in the theorem. 
\end{proof}

In \cite{truong}, we proved avoidance of saddle points for some modifications of Backtracking GD. In particular, Theorem 1.3 in that paper concerns a version called Backtracking GD-New, which is constructed using local Lipschitz constants $L(x)$ for $\nabla f$. Part iv) of Theorem 1.3 in \cite{truong} needs the assumption that there is $L_0>0$ so that if $x$ is a non-isolated generalised saddle point of $f$, then the $L(x)\leq L_0$. While this condition is already more general than what used before (requiring that $f$ is in $C^{1,1}_L$), it can actually be removed. In fact, by using the same argument through Lindel\"off theorem, we need only care about $L(z)$ for a countable number of saddle points, and the arguments in \cite{truong} go through.

\end{document}